\documentclass{amsart}
\usepackage[margin=3cm]{geometry}
\usepackage{amsmath,amsfonts,amssymb,amsthm}
\usepackage{graphics}
\usepackage{epsfig, esint}
\usepackage{graphics,color}
\usepackage{paralist}

\RequirePackage[colorlinks,citecolor=blue,urlcolor=blue]{hyperref}
\newcommand{\norm}[1]{\left\lVert#1\right\rVert}
   \definecolor{labelkey}{gray}{.8}
   \definecolor{refkey}{gray}{.8}

\providecommand{\bfa}{{\bf a}}

\providecommand{\R}{\mathbb{R}}

\providecommand{\divv}{{\rm {div}}}

\providecommand{\bfa}{{\bf a}}

\newcommand{\e}{\varepsilon}

\newcommand{\step}[1]{\medskip\noindent\textbf{Step #1. }}
\newcommand{\substep}[1]{\medskip\noindent\textit{Substep #1. }}

\newcommand{\ignore}[1]{}

\newtheorem{definition}{Definition}
\newtheorem{proposition}{Proposition}
\newtheorem{theorem}{Theorem}
\newtheorem{remark}{Remark}
\newtheorem{lemma}{Lemma}

\newtheorem{assumption}{Assumption}

\author[P Bella]{Peter Bella}
\address{TU Dortmund\\ Fakult\"at f\"ur Mathematik\\ Lehrstuhl I\\ Vogelpothsweg 87\\44227 Dortmund, Germany.}
\email{peter.bella@math.tu-dortmund.de}
\author[M. Sch\"affner]{Mathias Sch\"affner}
\address{Technische Universit\"at Dortmund, Fakult\"at f\"ur Mathematik\\
 Vogelpothsweg 87,44227 Dortmund, Germany.}
\email{mathias.schaeffner@tu-dortmund.de}

\title[Lipschitz-bounds for scalar integral functionals]{Lipschitz bounds for integral functionals with $(p,q)$-growth conditions} 

\begin{document}
\maketitle


\begin{abstract}
We study local regularity properties of local minimizer of scalar integral functionals of the form 
$$
\mathcal F[u]:=\int_\Omega F(\nabla u)-f u\,dx
$$ 
where the convex integrand $F$ satisfies controlled $(p,q)$-growth conditions. We establish Lipschitz continuity under sharp assumptions on the forcing term $f$ and improved assumptions on the growth conditions on $F$ with respect to the existing literature. Along the way, we establish an $L^\infty$-$L^2$-estimate for solutions of linear uniformly elliptic equations in divergence form which is optimal with respect to the ellipticity contrast of the coefficients.
\end{abstract}

\section{Introduction}

In this  note, we revisit the question of Lipschitz-regularity for local minimizers of integral functionals of the form
\begin{equation}\label{eq:int}
w\mapsto \mathcal F(w;\Omega):=\int_\Omega F(\nabla w)-fw\,dx.
\end{equation}
We recall in the case $F(z)=\frac12|z|^2$ local minimizer of \eqref{eq:int} satisfy $-\Delta w=f$. A classic theorem due to Stein \cite{stein} implies
\begin{equation}\label{eq:Deltaln1}
-\Delta w=f\in L^{n,1}(\Omega)\qquad\Rightarrow\qquad \nabla w\in L^\infty_{\rm loc}(\Omega),
\end{equation}
where $L^{n,1}(\Omega)$ denotes the Lorentz space (see below for a definition). In view of \cite{Cianchi92} the condition $f\in L^{n,1}(\Omega)$ is optimal in the Lorentz-space scale for the conclusion in \eqref{eq:Deltaln1}. In the last decade the implication in \eqref{eq:Deltaln1} was greatly generalized by replacing  the linear operator $\Delta$ by possibly degenerate/singular uniformly elliptic nonlinear operators, see \cite{B16,KM13,KM14,CianMaz11}. More recently, those results where extended to a wide range of non-uniformly elliptic variational problems by Beck and Mingione in \cite{BM20} (see also \cite{DM21,Filippis21} for related results for non-autonomous or non-convex vector valued problems).

In this paper, we are interested in a specific class of non-uniformly elliptic problems, namely functionals with so-called $(p,q)$-growth conditions which are described in the following

\begin{assumption}\label{ass}
Let $0<\nu\leq \Lambda<\infty$, $1<p\leq q<\infty$ and $\mu\in[0,1]$ be given. Suppose that $F:\R^n\to[0,\infty)$ is convex, locally $C^2$-regular in $\R^n\setminus\{0\}$ and satisfies
\begin{equation}\label{ass:Fpq}
\begin{cases}
\nu(\mu^2+|z|^2)^\frac{p}2\leq F(z)\leq \Lambda(\mu^2+|z|^2)^\frac{q}2+\Lambda (\mu^2+|z|^2)^\frac{p}2\\
|\partial^2 F(z)|\leq \Lambda(\mu^2+|z|^2)^\frac{q-2}2+\Lambda (\mu^2+|z|^2)^\frac{p-2}2,\\
\nu(\mu^2+|z|^2)^\frac{p-2}2|\xi|^2\leq \langle \partial^2 F(z)\xi,\xi\rangle,
\end{cases}
\end{equation}
for every choice of $z,\xi\in\R^n$ with $|z|>0$.
\end{assumption}
Regularity properties of local minimizers of \eqref{eq:int} in the case $p=q$ are classical, see, e.g.,\ \cite{Giu}. A systematic regularity theory in the case $p<q$ was initiated by Marcellini in \cite{Mar89,Mar91}, see \cite{MR21} for a recent overview.

%
%
%
%
%
%

Before we state our main result, we recall a standard notion of local minimality in the context of integral functionals with $(p,q)$-growth 
\begin{definition}\label{def:localmin}
We call $u\in W_{\rm loc}^{1,1}(\Omega)$ a local minimizer of  $\mathcal F$ given in \eqref{eq:int} with $f\in L_{\rm loc}^n(\Omega)$ if for every open set $\Omega'\Subset\Omega$ the following is true: 
\begin{equation*}
 \mathcal F(u,\Omega')<\infty
\end{equation*}
and
\begin{equation*}
 \mathcal F(u,\Omega')\leq \mathcal F(u+\varphi,\Omega')
\end{equation*}
for any $\varphi\in W^{1,1}(\Omega)$ satisfying ${\rm supp}\;\varphi\Subset \Omega'$.
\end{definition}
The main result of the present paper is

\begin{theorem}\label{T:1}
Let $\Omega\subset\R^n$, $n\geq3$ be an open bounded domain and suppose Assumption~\ref{ass} is satisfied with $1<p<q<\infty$ such that
\begin{equation}\label{eq:pqrhs}
\frac{q}p<1+\min\biggl\{\frac{2}{n-1},\frac{4(p-1)}{p(n-3)}\biggr\}.
\end{equation}
Let $u\in W_{\rm loc}^{1,1}(\Omega)$ be a local minimizer of the functional $\mathcal F$ given in \eqref{eq:int} with $f\in L^{n,1}(\Omega)$. Then $\nabla u$ is locally bounded in $\Omega$. When $p\geq 2-\frac{4}{n+1}$ or when $f\equiv0$ condition \eqref{eq:pqrhs} can be replaced by 
\begin{equation}\label{eq:pq}
\frac{q}p<1+\frac2{n-1}.
\end{equation}
\end{theorem}

\begin{remark}
Theorem~\ref{T:1} should be compared to the findings of the recent papers \cite{BM20} and \cite{BS19c}:  In \cite{BM20}, Beck and Mingione proved (among many other things) the conclusion of Theorem~\ref{T:1} under the more restrictive relation 
$$
\frac{q}p<1+\min\biggl\{\frac2n,\frac{4(p-1)}{p(n-2)}\biggr\}.
$$
Hence, we obtain here an improvement in the gap conditions on $\frac{q}p$. In \cite{BS19c}, we proved Theorem~\ref{T:1} in the specific case $f\equiv 0$, $p\geq2$ and $\mu=1$.
\end{remark}

The proof of Theorem~\ref{T:1} closely follows the strategy presented in \cite{BM20} and relies on careful estimates for certain \textit{uniformly elliptic} problems. To illustrate this, let us consider the case that $F$ satisfies Assumption~\ref{ass} with $p=2<q$ and $f\equiv0$: Let $u$ be a local minimizer of $\mathcal F(\cdot,B_1)$ and assume that $u$ is smooth. Standard arguments yield
\begin{equation}\label{intro:lineq}
\divv (A(x)\nabla \partial_i u)=0\qquad\mbox{where}\qquad A:=\partial^2F(\nabla u).
\end{equation}
Hence, $\partial_i u$ satisfies a linear elliptic equation where the ellipticity ratio  $\mathcal R(x)$ of the coefficients, that is the quotient of the highest and lowest eigenvalue of $A(x)$,  is determined by the size of $|\nabla u(x)|$. More precisely, \eqref{ass:Fpq} with $p=2<q$ implies $\mathcal R(x)\sim (1+|\nabla u(x)|^{q-2})$. By standard theory for uniformly elliptic equations applied to \eqref{intro:lineq}, there exists an exponent $m=m(n)>0$ such that
\begin{equation}\label{intro:linftyl2}
\|\partial_i u\|_{L^\infty(B_\frac12)}\lesssim \|\mathcal R\|_{L^\infty(B_1)}^m\|\partial_i u\|_{L^2(B_1)}\lesssim(1+ \|\nabla u\|_{L^\infty(B_1)}^{q-2})^{m}\|\partial_i u\|_{L^2(B_1)}.
\end{equation}
Appealing to some well-known iteration arguments, it is possible to absorb the $(1+ \|\nabla u\|_{L^\infty(B_1)}^{q-2})^{m}$-prefactor on the right-hand side in \eqref{intro:linftyl2} provided that $(q-2)m<1$, which yields the restriction $\frac{q}2<1+\frac{1}{2m}$. Once an a priori Lipschitz estimate for smooth minimizer is established, the proof of Theorem~\ref{T:1} follows by a careful regularization and approximation procedure (as in e.g.\ \cite{BM20}).

The main technical achievement of the present manuscript -- using a method introduced in \cite{BS19a} -- is an improvement of the exponent $m$ in \eqref{intro:linftyl2} compared to the previous results. This improvement is \textit{optimal} for $n\geq 4$ and essentially optimal for $n=3$. More precisely, we have

\begin{proposition}\label{P:2}
Let $B=B_R(x_0)\subset\R^n$, $n\geq3$ and $\kappa\in(0,\frac12)$. There exists $c=c(n,\kappa)<\infty$, where $c=c(n)$ provided $n\geq4$, such that the following is true. Let $0<\nu\leq \lambda<\infty$ and suppose $a\in L^\infty(B;\R^{n\times n})$ satisfies that $a(x)$ is symmetric for almost every $x\in B$ and  uniformly elliptic in the sense
\begin{equation}\label{def:ellipt}
\nu|z|^2\leq a(x)z\cdot z\leq \Lambda|z|^2\quad\mbox{for every $z\in\R^n$ and almost every $x\in B$.}
\end{equation}
Let $v\in W^{1,2}(B)$ be a subsolution, that is it satisfies
\begin{equation}\label{def:subsolution}
\int_{B}a\nabla v\cdot\nabla \varphi\leq0\qquad\mbox{for all $\varphi\in C_c^1(B)$ with $\varphi\geq0$.}
\end{equation}
Then,
\begin{equation}\label{est:P1}
\sup_{\frac12 B} v\leq c\biggl(\frac{\Lambda}{\nu}\biggr)^m\biggl(\fint_{B}(v_+)^2\,dx\biggr)^\frac12,\qquad\mbox{where}\quad m:=\frac12\begin{cases}\frac{n-1}2&\mbox{for $n\geq4$}\\1+\kappa&\mbox{for $n=3$}\end{cases}.
\end{equation}
Moreover, for $n\geq4$ the exponent $m=\frac{n-1}4$ is optimal for the estimate in \eqref{est:P1} and for $n=3$ the exponent is essentially optimal in the sense that the estimate in \eqref{est:P1} is in general false for $m<\frac12$.
\end{proposition}

While estimate \eqref{est:P1} is in some sense optimal, the condition \eqref{eq:pqrhs} in Theorem~\ref{T:1} is in general not optimal. To see this, we recall a result of \cite{BF}: Suppose Assumption~\ref{ass} is satisfied with $\mu=1$ and $F(0)=0$, $\partial F(0)=0$, then \textit{bounded} local minimizer of \eqref{eq:int} with $f\equiv0$ are locally Lipschitz provided $q<p+2$. This can be combined with the recent local boundedness \cite{HS19}, where it is proven that under Assumption~\ref{ass} local minimizers of \eqref{eq:int} (with $f=0$) are locally bounded provided $\frac1p-\frac1q\leq \frac1{n-1}$ and this condition is sharp in view of  \cite[Theorem 6.1]{Mar91}. Combining Theorem~\ref{T:1} with the above mentioned results of \cite{BF,HS19}, we deduce that Assumption~\ref{ass} (together with some mild technical extra assumptions) with $1<p\leq q$ and
\begin{equation}\label{eq:pqnew}
\frac{q}p<1+\max\biggl\{\frac{2}{n-1},\min\biggl\{\frac{2}{p},\frac{p}{n-1-p}\biggr\}\biggr\}
\end{equation}
implies that local minimizer of \eqref{eq:int} with $f\equiv 0$ are locally Lipschitz (see also \cite{AT21} for related discussion). While it is not clear whether \eqref{eq:pqnew} is optimal, it strictly improves condition \eqref{eq:pq} for example in the cases $p=3$ and $n\geq5$. Let us mention that condition \eqref{eq:pq} also appears in \cite{S21} in the context of higher integrability results for vectorial problems (where Lipschitz regularity fails even in the case $p=q$, see \cite{SY02}). While also in that case bounded minimizer enjoy higher gradient integrability under the condition $q<p+2$, see \cite{CKP11}, this result cannot be combined with an a priori local boundedness result which fails in the vectorial case already for $p=q$.

For non-autonomous functionals, i.e., $\int_\Omega f(x,Du)\,dx$, rather precise sufficient \& necessary conditions are established in \cite{ELM04}, where the conditions on $p,q$ and $n$ have to be balanced with the (H\"older)-regularity in space of the integrand. Currently, regularity theory for non-autonomous integrands with non-standard growth, e.g.\ $p(x)$-Laplacian or double phase functionals, are a very active field of research, see, e.g.,\  the recent papers \cite{BCM18,BR20,BO20,CMMP21,CFK20,CM15,DM19,DM21,DM22,EMM19,HHT17,HO21,Koch,MaPa21} and \cite{BDS20,BGS21} for related results about the Lavrentiev phenomena. 


\section{Preliminaries}


\subsection{Preliminary lemmata}

A crucial technical ingredient in the proof of Theorem~\ref{T:1} is the following lemma which can be found in \cite[Lemma~3]{BS19c}.
\begin{lemma}[\cite{BS19c}]\label{L:optimcutoff}
Fix $n\geq2$. For given $0<\rho<\sigma<\infty$ and $v\in L^1(B_\sigma)$ consider
\begin{equation*}
 J(\rho,\sigma,v):=\inf\left\{\int_{B_\sigma}|v||\nabla \eta|^2\,dx \;|\;\eta\in C_0^1(B_\sigma),\,\eta\geq0,\,\eta=1\mbox{ in $B_\rho$}\right\}.
\end{equation*}
Then for every $\delta\in(0,1]$ 
\begin{equation}\label{1dmin}
 J(\rho,\sigma,v)\leq  (\sigma-\rho)^{-(1+\frac1\delta)} \biggl(\int_{\rho}^\sigma \left(\int_{S_r} |v|\,d\mathcal H^{n-1}\right)^\delta\,dr\biggr)^\frac1\delta,
\end{equation}
where
$$
S_r:=\{x\,:\,|x|=r\}
$$
\end{lemma}

Moreover, we recall here the following classical iteration lemma
\begin{lemma}[Lemma~6.1, \cite{Giu}]\label{L:holefilling}
Let $Z(t)$ be a bounded non-negative function in the interval $[\rho,\sigma]$. Assume that for every $\rho\leq s<t\leq \sigma$ it holds
\begin{equation*}
 Z(s)\leq \theta Z(t)+(t-s)^{-\alpha} A+B,
\end{equation*}
with $A,B\geq0$, $\alpha>0$ and $\theta\in[0,1)$. Then, there exists $c=c(\alpha,\theta)\in[1,\infty)$ such that
\begin{equation*}
 Z(s)\leq c((t-s)^{-\alpha} A+B).
\end{equation*}
\end{lemma}

\subsection{Non-increasing rearrangement and Lorentz-spaces}

We recall the definition and useful properties of the non-increasing rearrangement $f^*$ of a measurable function $f$ and Lorentz spaces, see e.g.\ \cite[Section~22]{TatarBook}. For a measurable function $f:\R^n\to\R$, the non-increasing rearrangement is defined by
\begin{equation*}
 f^*(t):=\inf\{\sigma\in(0,\infty)\,:\,|\{x\in\R^n\,:\,|f(x)|>\sigma\}|\leq t\}.
\end{equation*}
Let $f:\R^n\to\R$ be a measurable function with ${\rm supp}f\subset \Omega$, then it holds for all $p\in[1,\infty)$
\begin{equation}\label{eq:ff*}
\int_\Omega |f(x)|^p\,dx=\int_0^{|\Omega|}(f^*(t))^p\,dt.
\end{equation}
A simple consequence of \eqref{eq:ff*} and the fact $f\leq g$ implies $f^*\leq g^*$ is the following inequality
\begin{equation}\label{est:omegat}
\sup_{|A|\leq t\atop A\subset\Omega}\int_A|f(x)|^p\leq \int_0^t(f_\Omega^*(t))^p\,dt,
\end{equation}
where $f_\Omega^*$ denotes the non-increasing rearrangement of $f_\Omega:=f\chi_\Omega$ (inequality \eqref{est:omegat} is in fact an \textit{equality} but for our purpose the upper bound suffices). 

The Lorentz space $L^{n,1}(\R^d)$ can be defined as the space of measurable functions $f:\R^n\to\R$ satisfying
\begin{equation*}
 \|f\|_{L^{n,1}(\R^n)}:=\int_0^\infty t^\frac1n f^*(t)\,\frac{dt}t<\infty.
\end{equation*}
Moreover, for $\Omega\subset\R^n$ and a measurable function $f:\R^n\to\R$, we set
\begin{equation*}
 \|f\|_{L^{n,1}(\Omega)}:=\int_0^{|\Omega|} t^\frac1n f_\Omega^*(t)\,\frac{dt}t<\infty,
\end{equation*}
where $f_\Omega$ is defined as above. Let us recall that $L^{n+\e}(\Omega)\subset L^{n,1}(\Omega)\subset L^n(\Omega)$ for every $\e>0$, where $L^{n,1}(\Omega)$ is the space of all measurable functions $f:\Omega\to\R$ satisfying $ \|f\|_{L^{n,1}(\Omega)}<\infty$ (here we identify $f$ with its extension by zero to $\R^n\setminus \Omega$). 


\section{Nonlinear iteration lemma and proof of Proposition~\ref{P:2}}\label{Sec:nonlineariteration}

In this section, we provide a nonlinear iteration lemma which will eventually be the main workhorse in the proof of Theorem~\ref{T:1}:

\begin{lemma}\label{L:basiciteration}
Let $B=B_R(x_0)\subset\R^n$, $n\geq3$, $\kappa\in(0,\frac12)$ and let $v\in W^{1,2}(B)\cap C(B)$ be non-negative and $f\in L^2(\R^n)$. Suppose there exists $ M_1\geq1$, $M_2,c_m>0$ and $k_0\geq0$ such that for all $k\geq k_0$ and for all $\eta\in C_c^1(B)$ with $\eta\geq0$ it holds
\begin{align}\label{L:basic:caccio}
\int_{B}|\nabla(v-k)_+|^2\eta^2\,dx\leq c_m^2M_1^2\int_{B}(v-k)_+^2|\nabla \eta|^2\,dx+c_m^2M_2^2\int_{B\cap\{v>k\}}\eta^2|f|^2.
\end{align}
Then there exist $c=c(c_m,n)\in[1,\infty)$ for $n\geq4$ and $c=c(c_m,\kappa)\in[1,\infty)$ for $n=3$ such that
\begin{equation}\label{L:basic:claim}
\|(v)_+\|_{L^\infty(\frac14 B)}\leq k_0+cM_1^{1+\max\{\kappa,\frac{n-3}2\}}\biggl(\fint_B (v-k_0)_+^2\,dx\biggr)^\frac12+cM_1^{\max\{\kappa,\frac{n-3}2\}}M_2 \|f\|_{L^{n,1}(B)}.
\end{equation}


\end{lemma}


\begin{remark}\label{rem:Loptim}[Optimality]
The exponent in the factor $M_1^{1+\max\{\kappa,\frac{n-3}2\}}$ is optimal in dimensions $n\geq4$ and almost optimal for $n=3$. Indeed, consider $v(x):=x_n^2+1-\Lambda |x'|^2$, where $x=(x_1,\dots,x_{n-1},x_n)=:(x',x_n)$ which clearly satisfies
\begin{equation}\label{eq:counterexample}
-\nabla\cdot a\nabla v=0\qquad\mbox{where}\qquad a:={\rm diag}(1,\dots,1,(n-1)\Lambda).
\end{equation}
Hence, by classical computations (using the symmetry of $a$) $v$ satisfies the Caccioppoli inequality \eqref{L:basic:caccio} with
\begin{equation}\label{eq:remoptim}
c_m^2=4(n-1),\qquad M_1=\Lambda^\frac12,\qquad f\equiv0.
\end{equation}
Obviously, we have $\sup_{B_\frac14}\, v\geq v(0)\geq1$ and
\begin{align*}
\fint_{B_1}(v_+)^2\lesssim& \int_{-1}^1\int_0^{\Lambda^{-1/2}(1+x_n^2)^{1/2}}r^{n-2}(x_n^2+1)^2+\Lambda^2r^{n+2}\,dr\,dx_n\\
\lesssim&\Lambda^{-\frac{n-1}2}\int_{-1}^1(1+x_n^2)^{\frac{n+3}2}\lesssim \Lambda^{-\frac{n-1}2},
\end{align*}
where $\lesssim$ means $\leq$ up to a multiplicative constant depending only on $n$. In particular, we have
$$
\frac{\|v_+\|_{L^\infty(B_\frac14)}}{\|v_+\|_{L^2(B_1)}}\geq c(n)\Lambda^{\frac{n-1}{4}}\stackrel{\eqref{eq:remoptim}}=c(n)M_1^{\frac{n-1}{2}}
$$
which matches exactly the scaling in \eqref{L:basic:claim} for $n\geq4$ and almost matches in the case $n=3$. We mention that $-(v)_+$ appeared already in \cite{connor} in the context of optimal dependencies in Krylov-Safanov estimates.
\end{remark}

\begin{remark} Lemma~\ref{L:basiciteration} should be compared with \cite[Lemma~3.1]{BM20}, where starting from a similar Caccioppoli inequality as \eqref{L:basiciteration} a pointwise bound for $v$ in terms of the $L^2$-norm $v$ and the Riesz potential of $f$ is deduced. The main improvement of Lemma~\ref{L:basiciteration} compared to \cite[Lemma~3.1]{BM20} lies in the dependence of $M_1$ (from $M_1^{\max\{\kappa,\frac{n-2}2\}}$ in \cite{BM20} to $M_1^{\max\{\kappa,\frac{n-3}2\}}$) which in view of Remark~\ref{rem:Loptim} is essentially optimal. Let us remark that variants of \cite[Lemma~3.1]{BM20} also play an important role in \cite{DM21,DM22} where non-autonomus functionals are considered.
\end{remark}

Before, we prove Lemma~\ref{L:basiciteration} let us observe that Proposition~\ref{P:2} is implied by Lemma~\ref{L:basiciteration}.

\begin{proof}[Proof of Proposition~\ref{P:2}]
By density arguments, we are allowed to use $\varphi=\eta^2(v-k)_+$, $k\geq0$ and $\eta\in C_c^1(B)$ with $\eta\geq0$ in \eqref{def:subsolution} which implies
\begin{equation*}
 \int_{B}a\nabla (u-k)_+\cdot \nabla (v-k)_+\eta^2\,dx\leq-2\int_Ba\nabla (v-k)_+\cdot (v-k)_+\eta\nabla \eta\,dx.
\end{equation*}
Assumption \eqref{def:ellipt} combined with Cauchy-Schwarz and Youngs inequality imply \eqref{L:basic:caccio} with $c_m^2=4$ and $M_1^2=\frac{\Lambda}\nu$. The claimed estimate \eqref{est:P1} now follows directly from Lemma~\ref{L:basiciteration} and the claimed optimality is a consequence of the discussion in Remark~\ref{rem:Loptim}.

\end{proof}

\begin{proof}[Proof of Lemma~\ref{L:basiciteration}]

Without loss of generality, we only consider the case $B=B_2$.

Throughout the proof we set
\begin{equation}\label{def:2star}
2^*=2^*(n,\kappa):=\begin{cases}\frac{2(n-1)}{n-3}&\mbox{if $n\geq4$}\\2+\frac2\kappa&\mbox{if $n=3$}\end{cases}.
\end{equation}

\step 1 Optimization Argument. We claim that there exists $c_1=c_1(n,2^*)\in[1,\infty)$ such that for all $k>h\geq0$ and all $\frac12\leq \rho<\sigma\leq 1$
\begin{equation}\label{L:basiciteratio:S1claim}
\|\nabla (v-k)_+\|_{L^2(B_\rho)}\leq \frac{c_1c_m M_1}{(\sigma-\rho)^{\alpha}} \frac{\norm{(v-h)_+}_{W^{1,2}(B_\sigma\setminus B_\rho)}^{\frac{2^*}2}}{(k-h)^{\frac{2^*}2-1}}+c_m M_2\omega_f(|A_{k,\sigma}|),
\end{equation}
where $\alpha=\frac12+\frac{2^*}4>0$, 
\begin{equation}\label{def:Alr}
A_{l,r}:=B_r\cap \{x\in\Omega\, :\,v(x)>l\}\qquad\mbox{for all $r>0$ and $l>0$,}
\end{equation}
and $\omega_f:[0,|B|]\to[0,\infty)$ is defined by
\begin{equation}\label{def:omegaf}
 \omega_f(t):=\biggl(\int_0^t((f\chi_B)^*(s))^2\,ds\biggr)^\frac12.
\end{equation}
We optimize the right-hand side of \eqref{L:basic:caccio} with respect to $\eta$ satisfying $\eta\in C_0^1(B_\sigma)$ and $\eta=1$ in $B_\rho$: we use Lemma~\ref{L:optimcutoff} and $1\leq\frac{v(x)-h}{k-h}$ whenever $v(x)\geq k$ in the form
\begin{align}\label{est:opteta}
    &\inf_{\eta\in\mathcal A_{\rho,\sigma}}\int_{B_1}|\nabla \eta|^2(v-k)_+^2\,dx\notag\\
    \leq&\frac1{(\sigma-\rho)^{1+\frac1\delta}}\biggl(\int_\rho^\sigma \biggl(\int_{S_r\cap A_{k,\sigma}}(v-k)_+^2\,d\mathcal H^{n-1}\biggr)^\delta\,dr\biggr)^\frac1\delta\notag\\
    \leq&\frac1{(\sigma-\rho)^{1+\frac1\delta}}\frac1{(k-h)^{2^*-2}}\biggl(\int_\rho^\sigma \biggl(\int_{S_r\cap A_{k,\sigma}}(v-h)_+^{2^*}\,d\mathcal H^{n-1}\biggr)^\delta\,dr\biggr)^\frac1\delta
\end{align}
for every $\delta>0$. Appealing to Sobolev inequality on spheres, we find $c=c(n,2^*)\in[1,\infty)$ such that for almost every $r\in[\rho,\sigma]$ it holds
\begin{equation}\label{est:sobsphere}
    \biggl(\int_{S_r}(v-h)_+^{2^*}\,d\mathcal H^{n-1}\biggr)^\frac{1}{2^*}\leq c\biggl(\int_{S_r}(v-h)_+^2+|\nabla (v-h)_+|^2\,d\mathcal H^{n-1}\biggr)^\frac{1}{2}.
\end{equation}
Inserting \eqref{est:sobsphere} in \eqref{est:opteta} with $\delta=\frac{2}{2^*}$, we obtain
\begin{equation}\label{est:opteta1}
    \inf_{\eta\in\mathcal A_{\rho,\sigma}}\int_{B_1}(v-k)_+^2|\nabla \eta|^2\,dx\leq \frac{c}{(\sigma-\rho)^{2\alpha}} \frac{\norm{(v-h)_+}_{W^{1,2}(B_\sigma\setminus B_\rho)}^{2^*-2}}{(k-h)^{2^*-2}}\norm{(v-h)_+}_{W^{1,2}(B_\sigma\setminus B_\rho)}^2
\end{equation}
The claim \eqref{L:basiciteratio:S1claim} follows from \eqref{L:basic:caccio} and \eqref{est:opteta1} combined with
\begin{equation*}
\int_{B_1\cap\{v>k\}}\eta^2|f|^2\,dx\leq \sup\biggl\{\int_A |f|^2\,dx\,:\,|A| \leq |A_{k,\sigma}|,\, A\subset B_1\biggr\}\stackrel{\eqref{est:omegat}}{\leq}\int_0^{|A_{k,\sigma}|}((f\chi_B)^*(s))^2\,ds
\end{equation*}
and the definition of $\omega_f$, see \eqref{def:omegaf}.

\step 2 One-Step improvement.

We claim that there exists $c_2=c_2(n)\in[1,\infty)$ such that
\begin{equation}\label{est:onestepL}
J(k,\rho)\leq \frac{c_1c_m M_1J(h,\sigma)^{\frac{2^*}2-1}}{(k-h)^{\frac{2^*}{2}-1}}\frac{J(h,\sigma)}{(\sigma-\rho)^\alpha}+c_mM_2\omega_{f}\biggl(\frac{c_2 J(h,\sigma)^{2_n^*}}{(k-h)^{2_n^*}}\biggr)+\frac{c_2J(h,\sigma)^{1+\frac{2_n^*}n}}{(k-h)^\frac{2_n^*}n}
\end{equation}
where $2_n^*:=\frac{2n}{n-2}$, 
\begin{equation}
J(t,r):=\|(v-t)_+\|_{W^{1,2}(B_r)}\qquad \forall t\geq0,\, r\in(0,1]
\end{equation}
and $\omega_f$ is defined in \eqref{def:omegaf}. Note that $k-h<v-h$ on $A_{k,r}$ for every $r>0$ and thus with help of Sobolev inequality (this time on the $n$-dimensional ball $B_\sigma$)
\begin{equation}\label{est:Aksigma}
|A_{k,\sigma}|\leq \int_{A_{k,\sigma}}\biggl(\frac{v(x)-h}{k-h}\biggr)^{2_n^*}\,dx\leq \frac{\|(v-h)_+\|_{L^{2_n^*}(B_\sigma)}^{2_n^*}}{(k-h)^{2_n^*}}\leq c_2 \frac{J(h,\sigma)^{2_n^*}}{(k-h)^{2_n^*}},
\end{equation}
where $c_2=c_2(n)\in[1,\infty)$. Combining \eqref{est:Aksigma} with \eqref{L:basiciteratio:S1claim}, we obtain
\begin{equation}\label{est:onestep1}
\|\nabla (v-k)_+\|_{L^2(B_\rho)}\leq \frac{c_1c_m M_1}{(\sigma-\rho)^{\alpha}} \frac{J(h,\sigma)^{\frac{2^*}2}}{(k-h)^{\frac{2^*}2-1}}+c_m M_2\omega_{f}\biggl(\frac{c_2 J(h,\sigma)^{2_n^*}}{(k-h)^{2_n^*}}\biggr).
\end{equation}
It remains to estimate $\|(v-k)_+\|_{L^2(B_\rho)}$: A combination of H\"older inequality, Sobolev inequality and \eqref{est:Aksigma} yield
\begin{align}\label{est:onestep2}
 \| (v-k)_+\|_{L^2(B_\rho)}\leq \|(v-h)_+\|_{L^{2_n^*}(B_\sigma)}|A_{k,\sigma}|^{\frac{1}{n}}\leq c(n)  \frac{J(h,\sigma)^{1+\frac{2_n^*}n}}{(k-h)^\frac{2_n^*}n}.
 \end{align}
Combining \eqref{est:onestep1} and \eqref{est:onestep2}, we obtain \eqref{est:onestepL}.

\step 3 Iteration

For $k_0\geq0$ and a sequence $(\Delta_\ell)_{\ell\in\mathbb N}\subset [0,\infty)$ specified below, we set
\begin{equation}\label{def:kell}
 k_\ell:=k_0+\sum_{i=1}^\ell\Delta_i,\quad \sigma_\ell=\frac12+\frac1{2^{\ell+1}}.
\end{equation}
For every $\ell\in\mathbb N\cup\{0\}$, we set $J_\ell:=J(k_\ell,\sigma_\ell)$. From \eqref{est:onestepL}, we deduce for every $\ell\in\mathbb N$
\begin{align}\label{est:iteration}
J_\ell\leq c_1c_m M_12^{(\ell+1)\alpha}\biggl(\frac{J_{\ell-1}}{\Delta_\ell}\biggr)^{\frac{2^*}2-1}J_{\ell-1} +c_m M_2\omega_f\biggl(\frac{c_2 J_{\ell-1}^{2_n^*}}{(\Delta_\ell)^{2_n^*}}\biggr)+c_2\biggl(\frac{J_{\ell-1}}{\Delta_\ell}\biggr)^{\frac{2_n^*}n}J_{\ell-1},
\end{align}
where $c_1$ and $c_2$ are as in Step~2. Fix $\tau=\tau(n,\kappa)\in(0,\frac12)$ such that
\begin{equation}\label{def:tau}
(2\tau)^{\frac{2^*}2-1}=2^{-\alpha}.
\end{equation}
%
We claim that we can choose $\{\Delta_\ell\}_{\ell\in\mathbb N}$ satisfying
\begin{equation}\label{est:sumdeltaell}
\sum_{\ell\in\mathbb N}\Delta_\ell<\infty
\end{equation}
in such a way that
\begin{equation}\label{ass:iterationJell}
 J_\ell\leq \tau^\ell J_0\qquad\mbox{for all $\ell\in\mathbb N\cup\{0\}$}.
\end{equation}
Obviously, \eqref{def:kell}, \eqref{ass:iterationJell} and $\tau\in(0,1)$ yield boundedness of $v$ in $B_\frac12$.

For $\ell\in\mathbb N$, we set $\Delta_\ell=\Delta_\ell^{(1)}+\Delta_\ell^{(2)}+\Delta_\ell^{(3)}$ with
\begin{equation}\label{def:deltaell}
\Delta_\ell^{(1)}=\biggl(2^\alpha 3c_1 c_m M_1 \tau^{-(\frac{2^*}2-1)-1}\biggr)^\frac1{\frac{2^*}2-1}J_02^{-\ell},\qquad \Delta_\ell^{(3)}=\biggl(\frac{3 c_2}\tau\biggr)^{\frac{n}{2_n^*}}\tau^{\ell-1}J_0 
\end{equation}
and $\Delta_\ell^{(2)}$ being the smallest value such that
\begin{align}\label{def:deltaell1}
 c_m M_2\omega_f\biggl(\frac{c_2 J_{\ell-1}^{2_n^*}}{(\Delta_\ell^{(2)})^{2_n^*}}\biggr)\leq \frac13\tau^\ell J_0
\end{align}
is valid. The choice of $\tau$ (see \eqref{def:tau}), $\Delta_\ell$ and estimate \eqref{est:iteration} combined with a straightforward induction argument yield \eqref{ass:iterationJell}. Indeed, assuming $J_{\ell-1}\leq \tau^{\ell-1}J_0$, we obtain
\begin{align*}
J_\ell\leq& \frac{c_1c_m M_12^{(\ell+1)\alpha}}{2^\alpha 3c_1 c_m M_1 \tau^{-(\frac{2^*}2-1)-1}}\biggl(\tau^{\ell-1}2^{\ell}\biggr)^{\frac{2^*}2-1}\tau^{\ell-1}J_0+\frac13\tau^\ell J_0+\frac{c_2\tau}{3c_3}\tau^{\ell-1}J_0\\
=&\frac13\tau 2^{\alpha \ell}\biggl((2\tau)^\ell\biggr)^{\frac{2^*}2-1}\tau^{\ell-1}J_0+\frac23\tau^\ell J_0\stackrel{\eqref{def:tau}}=\tau^\ell J_0.
\end{align*}

Using $\sum_{\ell\in\mathbb N}(2^{-\ell}+\tau^\ell)<\infty$, we deduce from \eqref{def:deltaell1} and \eqref{def:deltaell}
\begin{align}\label{sum:deltaell}
 \sum_{\ell\in\mathbb N}\Delta_\ell \leq \sum_{\ell\in\mathbb N}\frac{c_2^{\frac1{2_n^*}}\tau^{\ell-1}J_0}{(\omega_f^{-1}(\frac{\tau^\ell J_0}{3c_m M_2}))^{\frac1{2_n^*}}}+c(1+M_1^\frac{1}{\frac{2^*}2-1}) J_0,
\end{align}
where $c=c(n,\kappa,c_m)\in[1,\infty)$. Next, we show that $f\in L^{n,1}(B_1)$ ensures that the first term on the right-hand side of \eqref{sum:deltaell} is bounded and thus \eqref{est:sumdeltaell} is valid. Indeed,
\begin{align}\label{sum:omega1}
\sum_{\ell\in\mathbb N}\frac{\tau^{\ell}J_0}{(\omega_f^{-1}(\frac{\tau^\ell J_0}{3c_m M_2}))^{\frac12-\frac1n}}\leq& \frac1\tau\int_1^\infty \frac{\tau^x J_0}{(\omega_f^{-1}(\frac{\tau^xJ_0}{3c_mM_2}))^{\frac1{2}-\frac1n}}\,dx\notag\\
=&\frac1{\tau|\log \tau|}\int_0^\tau\frac{tJ_0}{(\omega_f^{-1}(\frac{tJ_0}{3c_m M_2}))^{\frac12-\frac1n}}\,\frac{dt}t\notag\\
\leq&\frac{3c_mM_2}{\tau|\log\tau|}\int_0^{\omega_f^{-1}({\frac{\tau J_0}{3c_m M_2}})} \frac{\omega_f'(s)}{s^{\frac12-\frac1n}}\,ds.
\end{align}
Recall $\omega_f(t)=(\int_0^t((f\chi_{B_1})^*(s))^2\,ds)^\frac12$ and thus $\omega_f'(t)=\frac12 \frac1{\omega_f(t)}(f\chi_{B_1})^*(t))^2$. Since $(f\chi_{B_1})^*$ is non-increasing, we have $\omega_f(t)\geq t^\frac12(f\chi_{B_1})^*(t)$ and 
\begin{align}\label{sum:omega2}
\int_0^{\omega^{-1}({\frac{\tau J_0}{3c_mM_2}})} \frac{\omega_f'(s)}{s^{\frac12-\frac1n}}\,ds
\leq&\frac12\int_0^{\infty}s^{\frac1n}(f\chi_{B_1})^*(s)\,\frac{ds}s=\frac12\|f\|_{L^{n,1}(B_1)}.
\end{align}
Notice that \eqref{ass:iterationJell} implies
\begin{equation*}
\|(v-k_0-\sum_{\ell\in\mathbb N}\Delta_\ell)_+\|_{L^2(B_\frac12)}=0\quad \Rightarrow\quad \sup_{B_\frac12}v\leq k_0+\sum_{\ell\in\mathbb N}\Delta_\ell.
\end{equation*}
%
Hence, appealing to \eqref{sum:deltaell}-\eqref{sum:omega2} and $M_1\geq1$, we find $c=c(c_m,n,\kappa)\in[1,\infty)$ such that
\begin{equation*}
 \sup_{B_\frac12}v\leq k_0+ cM_1^\frac{1}{\frac{2^*}2-1}\|(u-k_0)_+\|_{W^{1,2}(B_1)}+cc_m M_2\|f\|_{L^{n,1}(B_1)}.
\end{equation*}
The claimed estimate \eqref{L:basic:claim} with $B=B_2$, follows by the definition of $2^*$, see \eqref{def:2star} and a further application of \eqref{L:basic:caccio} with $\eta\in C_c^1(B_2)$ with $\eta=1$ in $B_1$ with $|\nabla \eta|\lesssim1$.
%

\end{proof}

%

\section{A priori estimate for regularized integrand}

In this section, we derive a priori estimates for regular weak solutions $u\in W^{1,\infty}_{\rm loc}(B)$ of the equation
\begin{equation}\label{eq:aeps}
-\nabla \cdot \bfa_\e(\nabla u)=f\qquad \mbox{in $B\subset\R^n$, $f\in L^\infty(\R^n)$,}
\end{equation}
where $\bfa_\e:\R^n\to\R^n$ is (a possibly regularized version of) $\partial F$ and satisfies
\begin{assumption}\label{ass:reg}
Let $0<\nu\leq \Lambda<\infty$, $1<p\leq q<\infty$, $\mu\in[0,1]$ and $\e,T>0$ be given. Suppose that $\bfa_\e\in C^1_{\rm loc}(\R^n,\R^n)$ satisfies
\begin{equation}\label{ass:bfaeps}
\begin{cases}
\partial_i (\bfa_\e)_j(z)=\partial_j (\bfa_\e)_i(z)&\mbox{for all $z\in\R^n$, and $i,j\in\{1,\dots,n\}$}\\
g_{1}(|z|)|\xi|^2\leq \langle\partial \bfa_\e(z)\xi,\xi\rangle \leq g_{2,\e}(|z|)|\xi|^2&\mbox{for all $z,\xi\in\R^n$ with $|z|\geq T$}\\
\mbox{$\partial\bfa_\e(z)$ is strictly positive definite}&\mbox{for all $z\in\R^n$ with $|z|\leq 2T$}
\end{cases}
\end{equation}
where for all $s\geq0$ 
%
%
\begin{align*}
g_1(s):=\nu(\mu^2+s^2)^{\frac{p-2}2},\qquad g_{2,\e}(s):=\Lambda(\mu^2+s^2)^{\frac{p-2}2}+\Lambda (\mu^2+s^2)^{\frac{q-2}2}+\e\Lambda(1+s^2)^\frac{\min\{p-2,0\}}2
\end{align*}

\end{assumption}

%
Moreover, we introduce (following \cite{BM20}) the quantity
\begin{equation}
G_T(t):=\int_T^{\max\{t,T\}} g_1(s)s\,ds\qquad (T>0).
\end{equation}
\begin{lemma}[Caccioppoli inequality]\label{L:caccio}
Suppose Assumption~\ref{ass:reg} is satisfied for some $\e,T>0$. There exists $c=c(n)\in[1,\infty)$ such that the following is true: Let $u\in W^{1,\infty}_{\rm loc}(B)$ be a weak solution to \eqref{eq:aeps}. Then it holds for all $\eta\in C_c^1(B)$ 
\begin{align}\label{eq:caccioaprio}
\int_{B}|\nabla((G_T(|\nabla u|)-k)_+)|^2\eta^2\,dx\leq& c\int_{B}\frac{g_{2,\e}(|\nabla u|)}{g_1(|\nabla u|)}(G_T(|\nabla u|)-k)_+^2|\nabla \eta|^2\,dx\notag\\
&+c\int_{B\cap\{G_T(|\nabla u|)>k\}}\eta^2|\nabla u|^2 |f|^2\
\end{align}
\end{lemma}

The proof of Lemma~\ref{L:caccio} follows almost verbatim the proof of \cite[Lemma 4.5]{BM20} where \eqref{eq:caccioaprio} is proven for a specific choice of $\eta$. 

\begin{proof}

The following computations are essentially a translation of the proof of \cite[Lemma 4.5]{BM20} to the present situation.

\step 0 Preliminaries.

Since we suppose that $u$ is Lipschitz continuous and $\bfa_\e\in C_{\rm loc}^1(\R^n,\R^n)$ satisfies the uniform estimate $\langle\partial \bfa_\e(z)\xi,\xi\rangle\geq\nu_\e|\xi|^2$ for some $\nu_\e>0$ for all $z,\xi\in\R^n$ (which follows from \eqref{ass:bfaeps}), we obtain from standard regularity theory that

\begin{equation*}
u\in W_{\rm loc}^{2,2}(B),\quad u\in C_{\rm loc}^{1,\alpha}(B)\,\mbox{for some $\alpha\in(0,1)$},\quad \bfa_\e(\nabla u)\in W_{\rm loc}^{1,2}(B,\R^n).
\end{equation*} 
Hence, we can differentiate equation \eqref{eq:aeps} and obtain for all $s\in\{1,\dots,n\}$ and every $\varphi\in W_0^{1,2}(B)$ with compact support in $B$ that
\begin{equation}\label{L:aprio:eulerdiffesp}
\int_{B}\langle \partial \bfa_\e(\nabla u)\nabla \partial_s u,\nabla \varphi\rangle \,dx=-\int_B f\partial_s \varphi\,dx.
\end{equation}
We test \eqref{L:aprio:eulerdiffesp} with 
$$
\varphi:=\varphi_s:=\eta^2 (G_T(|\nabla u|)-k)_+\partial_s u,\qquad k\geq0,\quad\eta\in C_c^1(B).
$$
On the set $\{G_T(|\nabla u|)>k\}$ holds
\begin{equation}\label{L:caccio:nablaphis}
\nabla \varphi_s=\eta^2 (G_T(|\nabla u|)-k)_+\nabla \partial_s u+\eta^2 \partial_s u\nabla (G_T(|\nabla u|)-k)_++2\eta (G_T(|\nabla u|)-k)_+\partial_s u\nabla \eta
\end{equation}
and
\begin{equation}\label{eq:GTg1}
\nabla (G_T(|\nabla u|-k)_+=\nabla (G_T(|\nabla u|))=g_1(|\nabla u|)\sum_{s=1}^n\partial_s u\nabla \partial_s u\quad\mbox{and}\quad g_1(|\nabla u|)>0.
\end{equation}
We compute,
\begin{align*}
&\sum_{s=1}^n \int_B\langle\partial \bfa_\e(\nabla u)\nabla \partial_s u,\nabla \partial_s u\rangle (G_T(|\nabla u|)-k)_+\eta^2\,dx\\
&+\int_B g_1(|\nabla u|)^{-1}\langle \partial \bfa_\e(\nabla u)\nabla (G_T(|\nabla u|)-k)_+,\nabla (G_T(|\nabla u|)-k)_+)\rangle \eta^2\,dx\\
\stackrel{\eqref{eq:GTg1}}{=}&\sum_{s=1}^n \int_B\langle\partial \bfa_\e(\nabla u)\nabla \partial_s u,(G_T(|\nabla u|)-k)_+\nabla \partial_s u\rangle \eta^2\,dx\\
&+\sum_{s=1}^n\int_B \langle \partial \bfa_\e(\nabla u)\nabla \partial_s u,(\partial_s u)\nabla (G_T(|\nabla u|)-k)_+)\rangle \eta^2\,dx\\
\stackrel{\eqref{L:aprio:eulerdiffesp}}=&-2\sum_{s=1}^n \int_B\langle\partial \bfa_\e(\nabla u)\nabla \partial_s u,\nabla \eta\rangle (\partial_s u)(G_T(|\nabla u|)-k)_+\eta\,dx-\sum_{s=1}^n\int_B f\partial_s\varphi_s\,dx\\
\stackrel{\eqref{eq:GTg1}}=&-2 \int_Bg_1(|\nabla u|)^{-1}\langle\partial\bfa_\e(\nabla u)\nabla((G_T(|\nabla u|)-k)_+),\nabla \eta\rangle (G_T(|\nabla u|)-k)_+\eta\,dx-\sum_{s=1}^n\int_B f\partial_s\varphi_s\,dx.
\end{align*}
The symmetry of $\partial \bfa_\e$ and Cauchy-Schwarz inequality yield
\begin{align*}
&-2 \int_Bg_1(|\nabla u|)^{-1}\langle\partial\bfa_\e(\nabla u)\nabla((G_T(|\nabla u|)-k)_+),\nabla \eta\rangle (G_T(|\nabla u|)-k)_+\eta\,dx\\
\leq&\frac12\int_Bg_1(|\nabla u|)^{-1}\langle\partial\bfa_\e(\nabla u)\nabla((G_T(|\nabla u|)-k)_+),\nabla((G_T(|\nabla u|)-k)_+)\rangle\eta^2\,dx\\
&+2\int_Bg_1(|\nabla u|)^{-1}\langle\partial\bfa_\e(\nabla u)\nabla\eta,\nabla\eta\rangle((G_T(|\nabla u|)-k)_+)^2\,dx.
\end{align*}
The previous two (in)equalities combined with \eqref{ass:bfaeps} imply the existence of $c=c(n)<\infty$ such that
\begin{align}\label{L:caccio:almostfinal1}
&\int_B\biggl(g_1(|\nabla u|)(G_T(|\nabla u|)-k)_+|\nabla^2 u|^2+|\nabla (G_T(|\nabla u|)-k)_+|^2\biggr)\eta^2\,dx\notag\\
\leq&c\int_B\biggl(\frac{g_{2,\e}(\nabla u)}{g_1(|\nabla u|)}\biggr)(G_T(|\nabla u|)-k)_+^2|\nabla \eta|^2\,dx+c\sum_{s=1}^n\int_B|f||\nabla \varphi_s|\,dx.
\end{align}
Appealing to Youngs inequality and \eqref{L:caccio:nablaphis}, we estimate the last term on the right-hand side by
\begin{align}\label{L:caccio:almostfinal2}
&\sum_{s=1}^n\int_B|f||\nabla \varphi_s|\,dx\notag\\
\leq&\frac12 \int_B \biggl(g_1(|\nabla u|)(G_T(|\nabla u|)-k)_+|\nabla^2 u|^2+|\nabla (G_T(|\nabla u|)-k)_+|^2\biggr)\eta^2\,dx\notag\\
&+c\int_{B}((G_T(|\nabla u|)-k)_+)^2|\nabla \eta|^2\,dx\notag\\
&+c\int_{B\cap \{G_T(|\nabla u|)>k\}}|f|^2\big(g_1(|\nabla u|)^{-1}(G_T(|\nabla u|)-k)_++|\nabla u|^2\big)\eta^2\,dx,
\end{align}
where $c=c(n)\in[1,\infty)$. The monotonicity of $s\mapsto g_1(s)s$ yields $G_T(t)\leq g_1(t)t(t-T)\leq g_1(t)t^2$ and thus
\begin{align}\label{L:caccio:almostfinal3}
\int_{B\cap \{G_T(|\nabla u|)>k\}}|f|^2(g_1(|\nabla u|)^{-1}(G_T(|\nabla u|)-k)_++|\nabla u|^2)\eta^2\,dx\leq 2\int_{B\cap \{G_T(|\nabla u|)>k\}}|f|^2|\nabla u|^2\eta^2\,dx
\end{align}
Combining \eqref{L:caccio:almostfinal1}-\eqref{L:caccio:almostfinal3}, we obtain
\begin{align*}
&\int_B\biggl(g_1(|\nabla u|)(G_T(|\nabla u|)-k)_+|\nabla^2 u|^2+|\nabla (G_T(|\nabla u|)-k)_+|^2\biggr)\eta^2\,dx\\
\leq&c\int_B\biggl(\frac{g_{2,\e}(\nabla u)}{g_1(|\nabla u|)}+1\biggr)(G_T(|\nabla u|)-k)_+^2|\nabla \eta|^2\,dx+c\int_{B\cap \{G_T(|\nabla u|)>k\}}|f|^2|\nabla u|^2\eta^2\,dx
\end{align*}
and the claim follows from $1\leq \frac{g_{2,\e}(t)}{g_1(t)}$ for all $t>0$.
\end{proof}

Combining the Caccioppoli inequality of Lemma~\ref{L:caccio} with Lemma~\ref{L:basiciteration}, we obtain the following local $L^\infty$-bound on $G_T(|\nabla u|)$
\begin{lemma}\label{L:apriorilipschitz}
Suppose Assumption~\ref{ass:reg} is satisfied for some $\e,T\in(0,1]$ and let $u\in W^{1,\infty}_{\rm loc}(B)$ be a weak solution to \eqref{eq:aeps}. Set
\begin{equation}\label{def:gammavartheta}
\gamma:=\frac12\frac{q-p}p\max\biggl\{\kappa,\frac{n-3}2\biggr\}+\frac{q}{2p},\qquad\tilde \gamma:=\frac12\frac{q-p}p\max\biggl\{\kappa,\frac{n-3}2\biggr\}+\frac1p
\end{equation}
and suppose that $\gamma,\gamma'\in(0,1)$. Then there exists $c=c(\gamma,\kappa,\Lambda,n,p,q)\in[1,\infty)$ such that
\begin{align}\label{claim:L:apriorilipschitz}
\|G_T(|\nabla u|)\|_{L^\infty(\frac12B)}\leq& c(1+\tfrac{\e}{(\mu^2+T^2)^\frac{p-2}2})^{\max\{\kappa,\frac{n-3}2\}+1}\fint_B G_T(|\nabla u|)+c\biggl(\fint_BG_T(|\nabla u|)\biggr)^{\frac12\frac1{1-\gamma}}\notag\\
&+c(1+\tfrac{\e}{(\mu^2+T^2)^\frac{p-2}2})^{\frac12\max\{\kappa,\frac{n-3}2\}}((T^2+\mu^2)^\frac12+(T^2+\mu^2)^{\frac12\tilde \gamma p})\|f\|_{L^{n,1}(B)}\notag\\
&+c\biggl((1+\tfrac{\e}{(\mu^2+T^2)^\frac{p-2}2})^{\frac12\max\{\kappa,\frac{n-3}2\}}\|f\|_{L^{n,1}(B)}\biggr)^\frac{p}{p-1}\notag\\
&+c\|f\|_{L^{n,1}(B)}^\frac1{1-\tilde \gamma}.
\end{align}
\end{lemma}


\begin{remark}\label{rem:gamma}
Note that $1<p\leq q$ and $\kappa\in(0,\frac12)$ imply that $\gamma$ and $\tilde \gamma$ defined in \eqref{def:gammavartheta} are positive. Moreover, for $n\geq4$ relation \eqref{eq:pqrhs} imply $\gamma,\tilde \gamma<1$. Indeed, we have 
\begin{align*}
\frac{q}p<1+\frac2{n-1}\quad\Rightarrow&\quad \gamma=\frac12\frac{q-p}p\frac{n-3}2+\frac{q}{2p}<\frac12\frac{2}{n-1}\frac{n-3}2+\frac12+\frac1{n-1}=1\\
\frac{q}p<1+\frac{4(p-1)}{p(n-3)}\quad\Rightarrow&\quad \tilde \gamma=\frac12\frac{q-p}p\frac{n-3}2+\frac1p<\frac12\frac{4(p-1)}{p(n-3)}\frac{n-3}2+\frac1p=1.
\end{align*}
In dimension $n=3$ and $\kappa\in(0,\frac12)$, a straightforward computation yields that $\kappa<\frac{2p-q}{q-p}$ implies $\gamma<1$ and $\kappa<2\frac{p-1}{q-p}$ implies $\tilde \gamma<1$.
\end{remark}

\begin{proof}[Proof of Lemma~\ref{L:apriorilipschitz}] Throughout the proof we write $\lesssim$ if $\leq$ holds up to a multiplicative constant depending only on $\kappa,\Lambda,\nu,n,p$ and $q$.

\step 0 Technical estimates.

There exists $c_1=c_1(\nu,\Lambda,p,q)\in[1,\infty)$ such that for all $t\geq T>0$ holds

\begin{equation}\label{eq:relateg1gT}
\frac{g_{2,\e}(t)}{g_1(t)}\leq c_1 (G_T(t)^\frac{q-p}p+1+\frac{\e}{(\mu^2+T^2)^{\frac{p-2}2}})\quad\mbox{and}\quad t\leq c_1 G_T(t)^\frac1p+(\mu^2+T^2)^\frac12.
\end{equation}
To establish \eqref{eq:relateg1gT}, we first compute for all $t\geq T$
\begin{equation}\label{GTt}
G_T(t)=\nu\int_T^{t}(\mu^2+s^2)^\frac{p-2}2s\,ds=\frac\nu{p}(\mu^2+t^2)^\frac{p}2-\frac\nu{p}(\mu^2+T^2)^\frac{p}2
\end{equation}
and thus
\begin{align*}
\frac{g_{2,\e}(t)}{g_1(t)}=&\frac{\Lambda}\nu(\mu^2+t^2)^{\frac{q-p}2}+\frac{\Lambda}\nu+\e\frac{\Lambda}\nu\frac{(1+t^2)^\frac{\min\{p-2,0\}}2}{(\mu^2+t^2)^{\frac{p-2}2}}\\
\leq& \frac{\Lambda}\nu \biggl(\frac{p}\nu G_T(t)+(\mu^2+T^2)^\frac{p}2\biggr)^{\frac{q-p}p}+\frac\Lambda\nu+\e\frac{\Lambda}{\nu}(1+(\mu^2+T^2)^{-\frac{p-2}2}),
\end{align*}
which implies the first estimate of \eqref{eq:relateg1gT} (recall $\mu\in[0,1]$ and $\e,T\in(0,1]$). The second estimate of \eqref{eq:relateg1gT} follows from \eqref{GTt} in the form: For $t\geq T>0$ holds
$$
t^p\leq \frac{p}\nu G_T(t)+(\mu^2+T^2)^\frac{p}2
$$
and the second estimate of \eqref{eq:relateg1gT} follows by taking the $p$-th root.

\step 1 In this step, we suppose $B_1\Subset B$ and prove 
\begin{align}\label{est:keyapprio}
&\|G_T(|\nabla u|)\|_{L^\infty(B_\frac14)}\notag\\
\lesssim&\|G_T(|\nabla u|)\|_{L^\infty(B_1)}^\gamma\|G_T(|\nabla u|)\|_{L^1(B_1)}^\frac12\notag\\
&+(1+\tfrac{\e}{(\mu^2+T^2)^\frac{p-2}2})^{\frac12\max\{\kappa,\frac{n-3}2\}+\frac12}\|G_T(|\nabla u|)\|_{L^\infty(B_1)}^\frac12\|G_T(|\nabla u|)\|_{L^1(B_1)}^{\frac12}\notag\\
&+(1+\tfrac{\e}{(\mu^2+T^2)^\frac{p-2}2})^{\frac12\max\{\kappa,\frac{n-3}2\}}(\|G_T(|\nabla u|)\|_{L^\infty(B_1)}^\frac1p+(\mu^2+T^2)^\frac12)\|f\|_{L^{n,1}(B_1)}\notag\\
&+(\|G_T(|\nabla u|)\|_{L^\infty(B_1)}^{\tilde \gamma}+(\mu^2+T^2)^{\frac12 \tilde \gamma p})\|f\|_{L^{n,1}(B_1)}
\end{align}


%
where $\gamma$ and $\tilde \gamma$ are defined in \eqref{def:gammavartheta}.

%
A direct consequence of the Caccioppoli inequality of Lemma~\ref{L:caccio} and the iteration Lemma~\ref{L:basiciteration} with the choice
$$
M_1^2=\biggl\|\frac{g_{2,\e}(|\nabla u|)}{g_1(|\nabla u|)}\biggr\|_{L^\infty(B_1\cap\{|\nabla u| \geq T\})}\qquad\mbox{and}\qquad M_2^2=\|\nabla u\|_{L^\infty(B_1)}^2
$$
is the following Lipschitz estimate 
%
%
\begin{align}\label{est:keyapprio:prep1}
&\|G_T(|\nabla u|)\|_{L^\infty(B_\frac14)}\notag\\
\lesssim&\biggl(\biggl\|\frac{g_{2,\e}(|\nabla u|)}{g_1(|\nabla u|)}\biggr\|_{L^\infty(B_1\cap\{|\nabla u| \geq T\})}\biggr)^{\frac12+\frac12\max\{\kappa,\frac{n-3}2\}}\|G_T(|\nabla u|)\|_{L^{2}(B_1)}\notag\\
&+\biggl(\biggl\|\frac{g_{2,\e}(|\nabla u|)}{g_1(|\nabla u|)}\biggr\|_{L^\infty(B_1\cap\{|\nabla u| \geq T\})}\biggr)^{\frac12\max\{\kappa,\frac{n-3}2\}}\|\nabla u\|_{L^\infty(B_1 )}\|f\|_{L^{n,1}(B_1)}.
\end{align}
Estimate \eqref{est:keyapprio} follows from \eqref{est:keyapprio:prep1} in combination with  \eqref{eq:relateg1gT} in the form
\begin{align}\label{est:keyapprio:prep2}
\biggl\|\frac{g_{2,\e}(|\nabla u|)}{g_1(|\nabla u|)}\biggr\|_{L^\infty(B_1\cap\{|\nabla u| \geq T\})}\lesssim \|G_T(|\nabla u|)\|_{L^\infty(B_1)}^\frac{q-p}p+1+\frac{\e}{(\mu^2+T^2)^{\frac{p-2}2}}
\end{align}
and
\begin{equation*}
\|\nabla u\|_{L^\infty(B_1)}\leq c_1 \|G_T(|\nabla u|)\|_{L^\infty(B_1)}^\frac1p+(\mu^2+T^2)^\frac12,
\end{equation*}
and the elementary interpolation inequality $\|\cdot\|_{L^2}\leq (\|\cdot\|_{L^\infty}\|\cdot\|_{L^1})^\frac12$.

%
%
%

\step 2 Conclusion

Appealing to standard scaling and covering arguments, we deduce from Step~1 the following: For every $x_0\in B$ and $0<\rho<\sigma$ satisfying $B_\sigma(x_0)\Subset B$ it holds
\begin{align}\label{est:keyapprio1}
&\|G_T(|\nabla u|)\|_{L^\infty(B_\rho(x_0))}\notag\\
\lesssim&(\sigma-\rho)^{-\frac{n}2}\|G_T(|\nabla u|)\|_{L^\infty(B_\sigma(x_0))}^\gamma\|G_T(|\nabla u|)\|_{L^1(B_\sigma(x_0))}^\frac12\notag\\
&+(\sigma-\rho)^{-\frac{n}2}(1+\tfrac{\e}{(\mu^2+T^2)^\frac{p-2}2})^{\frac12\max\{\kappa,\frac{n-3}2\}+\frac12}\|G_T(|\nabla u|)\|_{L^\infty(B_\sigma(x_0))}^{\frac12}\|G_T(|\nabla u|)\|_{L^1(B_\sigma(x_0))}^{\frac12}\notag\\
&+(1+\tfrac{\e}{(\mu^2+T^2)^\frac{p-2}2})^{\frac12\max\{\kappa,\frac{n-3}2\}}(\|G_T(|\nabla u|)\|_{L^\infty(B_\sigma(x_0))}^\frac1p+(\mu^2+T^2)^\frac12)\|f\|_{L^{n,1}(B_\sigma(x_0))}\notag\\
&+(\|G_T(|\nabla u|)\|_{L^\infty(B_\sigma(x_0))}^{\tilde \gamma}+(\mu^2+T^2)^{\frac12 \tilde \gamma p})\|f\|_{L^{n,1}(B_\sigma(x_0))}
\end{align}

From estimate \eqref{est:keyapprio1} in combination with Young inequality and assumption $\gamma,\tilde \gamma\in(0,1)$, we obtain the existence of $c=c(\kappa,\nu,\Lambda,n,p,q)\in[1,\infty)$ such that
\begin{align*}
\|G_T(|\nabla u|)\|_{L^\infty(B_\rho(x_0))}\leq& \frac12\|G_T(|\nabla u|)\|_{L^\infty(B_\sigma(x_0))}+c\frac{\|G_T(|\nabla u|)\|_{L^1(B_\sigma(x_0))}^{\frac12\frac1{1-\gamma}}}{(\sigma-\rho)^{\frac1{1-\gamma}\frac{n}2}}\\
&+ c\frac{(1+\tfrac{\e}{(\mu^2+T^2)^{\frac{p-2}2}})^{\max\{\kappa,\frac{n-3}2\}+1}\|G_T(|\nabla u|)\|_{L^1(B_\sigma(x_0))}}{(\sigma-\rho)^{n}}\\
&+(1+\tfrac{\e}{(\mu^2+T^2)^\frac{p-2}2})^{\frac12\max\{\kappa,\frac{n-3}2\}}((T^2+\mu^2)^\frac12+(T^2+\mu^2)^{\frac12\tilde \gamma p})\|f\|_{L^{n,1}(B_\sigma(x_0))}\\
&+c\biggl((1+\tfrac{\e}{(\mu^2+T^2)^\frac{p-2}2})^{\frac12\max\{\kappa,\frac{n-3}2\}}\|f\|_{L^{n,1}(B_\sigma(x_0))}\biggr)^\frac{p}{p-1}\\
&+c\|f\|_{L^{n,1}(B_\sigma(x_0))}^\frac{1}{1-\tilde \gamma}.
\end{align*}
The claimed inequality \eqref{claim:L:apriorilipschitz} (for $B=B_1$) now follows from Lemma~\ref{L:holefilling}.
\end{proof}

\section{Proof of Theorem~\ref{T:1}}

In this section, we prove Theorem~\ref{T:1} together with a suitable gradient estimate. More precisely, we show the following result which obviously contains the statement of Theorem~\ref{T:1}
\begin{theorem}\label{T:2}
Let $\Omega\subset\R^n$, $n\geq3$ be an open bounded domain and suppose Assumption~\ref{ass} is satisfied with $1<p<q<\infty$ such that \eqref{eq:pqrhs}. Let $u\in W_{\rm loc}^{1,1}(\Omega)$ be a local minimizer of the functional $\mathcal F$ given in \eqref{eq:int} with $f\in L^{n,1}(\Omega)$. Then $\nabla u$ is locally bounded in $\Omega$.  Moreover, for every $\kappa\in(0,\min\{\frac12,\frac{2p-q}{q-p},2\frac{p-1}{q-p}\})$ there exists $c=c(\kappa,\Lambda,\nu,n,p,q)\in[1,\infty)$ such that for all $B\Subset \Omega$ it holds

\begin{align}\label{est:T1}
\|\nabla u\|_{L^\infty(\frac12 B)}\leq& c\biggl(\fint_B F(\nabla u)\,dx+\|f\|_{L^{n,1}(B)}^\frac{p}{p-1}\biggr)^\frac1p\notag\\
&+c\biggl(\fint_B F(\nabla u)\,dx+\|f\|_{L^{n,1}(B)}^\frac{p}{p-1}\biggr)^{\alpha_n}\notag\\
&+c\|f\|_{L^{n,1}(B)}^{\beta_n}
\end{align}

where

\begin{align}\label{def:alphabetan}
\alpha_n:=\begin{cases}\frac{2}{(n+1)p-(n-1)q}&\mbox{if $n\geq4$}\\ \frac1{2p-q-(q-p)\kappa}&\mbox{if $n=3$}\end{cases},\quad \beta_n:=\begin{cases}\frac{4}{4(p-1)-(q-p)(n-3)}&\mbox{if $n\geq4$}\\\frac2{2(p-1)-(q-p)\kappa}&\mbox{if $n=3$}\end{cases}
\end{align}
%
%
In the case $n\geq4$ the constant $c$ in \eqref{est:T1} is independent of $\kappa$. When $p\geq 2-\frac{4}{n+1}$ or when $f\equiv0$ condition \eqref{eq:pqrhs} can be replaced by \eqref{eq:pq}.
\end{theorem}

\begin{proof}[Proof of Theorem~\ref{T:2}]
Throughout the proof we write $\lesssim$ if $\leq$ holds up to a multiplicative constant that depends only on $\kappa,\Lambda,\nu,n,p$ and $q$. We assume that $B_2\Subset\Omega$ and show
\begin{align}\label{est:apriorilipschitzu0}
&\|\nabla u\|_{L^\infty(B_\frac12)}\notag\\
\lesssim&\biggl(\fint_{B_{1}}F(\nabla u)\,dx+\|f\|_{L^{n,1}(B_1)}^\frac{ p}{ p-1}\biggr)^\frac1p\notag\\
&+\biggl(\fint_{B_{1}}F(\nabla u)\,dx+\|f\|_{L^n(B_1)}^\frac{ p}{ p-1}\biggr)^{\frac1{2p}\frac1{1-\gamma}}+\|f\|_{L^{n,1}(B_1)}^{\frac1p\frac1{1-\tilde \gamma}},
\end{align}
where $\gamma$ and $\tilde\gamma$ are given in \eqref{def:gammavartheta}. Clearly, the conclusion follows from a standard scaling, translation and covering arguments using $\alpha_n={\frac1{2p}\frac1{1-\gamma}}$ and $\beta_n={\frac1p\frac1{1-\tilde \gamma}}$.

\step 0 Preliminaries.

Following \cite{BM20}, we introduce various regularizations on the minimizer $u$, the integrand $F$ and the forcing term $f$: For this we choose a decreasing sequence $(\e_m)_{m\in\mathbb N}\subset(0,1)$ satisfying $\e_m\to0$ as $m\to\infty$. We set $\overline u_m:=u\ast \varphi_{\e_m}$ with $\varphi_\e:=\e^{-n}\varphi(\frac{\cdot}\e)$ and $\varphi$ being a non-negative, radially symmetric mollifier, i.e. it satisfies
$$
\varphi\geq0,\quad {\rm supp}\; \varphi\subset B_1,\quad \int_{\R^n}\varphi(x)\,dx=1,\quad \varphi(\cdot)=\widetilde \varphi(|\cdot|)\quad \mbox{for some $\widetilde\varphi\in C^\infty(\R)$}.
$$
Moreover, we denote by $f_m$ the truncated forcing $f_m(x)=\min\{\max\{f(x),-m\},m\}$ and consider the functional
\begin{equation}
\mathcal F_{m}(w,B):=\int_B[F_{\e_m}(\nabla w)-f_mw]\,dx,
\end{equation}
where for all $z\in\R^n$
$$F_\e(z):=\widetilde F_{\e}(z)+\e L_p(z)\quad\mbox{with $L_p(z):=\frac12|z|^2$ for $p\geq2$ and $L_p(z):=(1+|z|^2)^\frac{p}2-1$ for $p\in(1,2)$}$$
and $\widetilde F_\e$ satisfies for all $\e\in(0,\e_0]$ with $\e_0=\e_0(F,T)\in(0,1]$
\begin{equation}\label{ass:tildeFeps}
\widetilde F_\e\geq0,\quad \mbox{$\widetilde F_\e\in C^2_{\rm loc}(\R^n)$ is convex and $\widetilde F_\e=F$ on $\R^n\setminus B_\frac{T}2$}.
\end{equation}
(obviously $\widetilde F_\e$ and thus $F_\e$ depends also on $T>0$ which is suppressed in the notation) and it holds
\begin{equation}\label{ass:tildeFeps1}
\limsup_{\e\to0}\sup_{z\leq T}|\widetilde F_\e(z)-F(z)|=0.
\end{equation}

 In the case $F\in C^2_{\rm loc}(\R^n)$, we simply set $\widetilde F_\e\equiv F$ and in the case that $F$ is singular at zero we give (a standard) smoothing and gluing construction in Step~3 below.

Clearly, the functionals $\mathcal F_{m}$ are strictly convex and we denote by $u_m\in W^{1,1}(B)$ the unique function satisfying
\begin{equation}
\mathcal F_m(u_m,B)\leq \mathcal F_m(v,B)\qquad\mbox{for all $v\in \overline u_m+W_0^{1,1}(B)$}
\end{equation}
Appealing to \cite[Theorem~4.10]{BM20} (based on \cite{BB16}), we have $u_m\in W_{\rm loc}^{1,\infty}(B)$. In particular it follows that $u_m$ satisfies the Euler-Lagrange equation
$$
-\divv (\partial F_{\e_m}(\nabla u_m))=f_m
$$
and since $\bfa_\e:=\partial F_{\e_m}$ satisfies Assumption~\ref{ass:reg} (with $\e=\e_m$) we can apply Lemma~\ref{L:apriorilipschitz}. Note that in view of Remark~\ref{rem:gamma}, the assumptions on $p,q$ and $\kappa$ ensure $\gamma,\tilde \gamma\in (0,1)$. 

\step 1 We claim that 
\begin{align}\label{est:apriorilipschitzum}
&\|\nabla u_m\|_{L^\infty(B_\frac12)}^p\notag\\
\lesssim&(1+\tfrac{\e_m}{(\mu^2+T^2)^\frac{p-2}2})^{\max\{\kappa,\frac{n-3}2\}+1}\biggl(\int_{B_{1+\e_m}}\widetilde F_{\e_m}(\nabla u)\,dx+\e_m\int_{B_1}L_p(\nabla \overline u_m)\,dx+\|f\|_{L^n(B_1)}^\frac{ p}{ p-1}+T^{ p}+\mu^{ p}\biggr)\notag\\
&+\biggl(\int_{B_{1+\e_m}}\widetilde F_{\e_m}(\nabla u)\,dx+\e_m\int_{B_1}L_p(\nabla \overline u_m)\,dx+\|f\|_{L^n(B_1)}^\frac{ p}{ p-1}+T^{ p}+\mu^{ p}\biggr)^{\frac12\frac1{1-\gamma}}\notag\\
&+(1+\tfrac{\e_m}{(\mu^2+T^2)^\frac{p-2}2})^{\frac12\max\{\kappa,\frac{n-3}2\}}((T^2+\mu^2)^\frac12+(T^2+\mu^2)^{\frac12\tilde \gamma p})\|f\|_{L^{n,1}(B_1)}\notag\\
&+\biggl((1+\tfrac{\e_m}{(\mu^2+T^2)^\frac{p-2}2})^{\frac12\max\{\kappa,\frac{n-3}2\}}\|f\|_{L^{n,1}(B_1)}\biggr)^\frac{p}{p-1}\notag\\
&+\|f\|_{L^{n,1}(B_1)}^\frac1{1-\tilde \gamma}
\end{align}
 and
\begin{equation}\label{est:eneum}
\int_{B_1}F_{\e_m}(\nabla u_{m})\,dx\lesssim \int_{B_{1+\e_m}}\widetilde F_{\e_m}(\nabla u)\,dx+\e_m\int_{B_1}L_p(\nabla \overline u_m)\,dx+\|f\|_{L^n(B_1)}^{\frac{ p}{ p-1}}+(\mu^2+T^2)^\frac{p}2.
\end{equation}

A combination of H\"older and Sobolev inequality with the elementary inequality 
$$
\nu |z|^p\leq \nu(\mu^2+T^2)^\frac{p}2+F_{\e_m}(z) 
$$ 
(which follows from the definition of $F_\e$, \eqref{ass:tildeFeps} and \eqref{ass:Fpq}) yields
\begin{align*}
\|f_m(u_{m}-\overline u_m)\|_{L^1(B_1)}\leq& \|f_m\|_{L^n(B_1)}\|u_m-\overline u_m\|_{L^{\frac{n}{n-1}}(B_1)}\\
\leq&c(n, p) \|f_m\|_{L^n(B_1)}\|\nabla (u_m-\overline u_m)\|_{L^{p}(B_1)}\\
\leq&c\|f_m\|_{L^n(B_1)}\biggl(\int_{B_1}F_{\e_m}(\nabla u_{m})+F_{\e_m}(\nabla \overline u_m)\,dx+(\mu^2+T^2)^\frac{p}2\biggr)^\frac1{ p}\\
\leq&\tfrac1{ p}\biggl(\int_{B_1}F_{\e_m}(\nabla u_{m})+F_{\e_m}(\nabla \overline u_m)\,dx+(\mu^2+T^2)^\frac{p}2\biggr)+(1-\tfrac1{ p})(c\|f_m\|_{L^n(B_1)})^\frac{ p}{ p -1}
\end{align*}
where $c=c(n,\nu,\Lambda,p,q)\in[1,\infty)$. Combining the above estimate with the minimality of $u_{m}$ and the convexity of $\widetilde F_\e$ in the form
\begin{align}\label{est:eneum2}
 \int_{B_1}F_{\e_m}(\nabla u_{m})\,dx\leq& \int_{B_1}F_{\e_m}(\nabla \overline u_m)-f_m(\overline u_m-u_{m})\,dx\notag\\
\leq& \int_{B_{1+\e_m}}\widetilde F_{\e_m}(\nabla u)\,dx+\int_{B_1}\e_mL_p(\nabla \overline u_m)-f_m(\overline u_m-u_{m})\,dx
\end{align}
we obtain \eqref{est:eneum}. The claimed Lipschitz-estimate \eqref{est:apriorilipschitzum} follows from Lemma~\ref{L:apriorilipschitz}, estimates \eqref{eq:relateg1gT}, \eqref{est:eneum} and
\begin{align*}
0\leq G_T(|\nabla u_m|)\stackrel{\eqref{GTt}}\lesssim& (\mu^2+|\nabla u_m|^2)^\frac{p}2\stackrel{\eqref{ass:Fpq}}{\lesssim} F_{\e_m}(\nabla u_m)+(\mu^2+T^2)^\frac{p}2.
\end{align*}

\step 2 Passing to the limit.

\substep{2.1}  We claim
\begin{equation}\label{claim:umem}
\lim_{m\to\infty}\e_m\int_{B_1}L_p(\nabla \overline u_m)\,dx=0
\end{equation}
and
\begin{equation}\label{claim:umem0}
\lim_{m\to\infty}\int_{B_{1+\e_m}}\widetilde F_{\e_m}(\nabla u)\,dx= \int_{B_1}F(\nabla u)\,dx.
\end{equation}
We first note that $F(\nabla u)\in L^1_{\rm loc}(B_2)$. Indeed, by Definition~\ref{def:localmin} combined with H\"older and Sobolev inequality, we have for every $\tilde B\Subset B_2$
\begin{align*}
\int_{\tilde B}F(\nabla u)\,dx=&\mathcal F(u,\tilde B)+\int_{\tilde B}fu\leq  \mathcal F(u,\tilde B)+\|f\|_{L^n(\tilde B)}\|u\|_{L^{\frac{n}{n-1}}(\tilde B)}\\
\lesssim& \mathcal F(u,\tilde B)+\|f\|_{L^n(\tilde B)}\|u\|_{W^{1,1}(\tilde B)}<\infty
\end{align*}
For $p\geq2$, equation \eqref{claim:umem} follows from
$$
\int_{B_1}L_p(\nabla\overline u_m)\,dx=\frac12\|\nabla \overline u_m\|_{L^2(B_1)}^2\lesssim \|\nabla u\|_{L^2(B_\frac32)}^2\lesssim \biggl(\int_{B_\frac32}F(\nabla u)\,dx\biggr)^\frac{2}{ p}<\infty.
$$
In the case $p\in(1,2)$, equation \eqref{claim:umem} is a consequence of 
$$
\|L_p(\nabla \overline u_m)\|_{L^1(B_1)}\lesssim 1+\|\nabla u\|_{L^p(B_\frac32)}^p\lesssim 1+\int_{B_\frac32}F(\nabla u)\,dx<\infty.
$$

The argument for \eqref{claim:umem0} follows from the identity 
$$
\int_{B_{1+\e_m}}\widetilde F_{\e_m}(\nabla u)\,dx= \int_{B_{1+\e_m}}F(\nabla u)\,dx+\int_{B_{1+\e_m}\cap \{|\nabla u|\leq T\}}\widetilde F_{\e_m}(\nabla u)-F(\nabla u)\,dx.
$$
together with the uniform convergence \eqref{ass:tildeFeps1} and $F(\nabla u)\in L^1_{\rm loc}(B_2)$.

\substep{2.2} Proof of \eqref{est:apriorilipschitzu0}.

From \eqref{est:apriorilipschitzum}, \eqref{est:eneum} and \eqref{claim:umem}, we deduce the existence of a subsequence and $\overline u\in u+W^{1,1}_0(B)$ such that
\begin{align*}
u_m\rightharpoonup \overline u&\qquad\mbox{weakly in $W^{1, p}(B_1)$}\\
u_m\rightharpoonup \overline u&\qquad\mbox{weakly$^*$ in $W^{1,\infty}(B_\frac12)$}
\end{align*}
In view of \eqref{est:apriorilipschitzum}, \eqref{claim:umem}, \eqref{claim:umem0} and the weak$^*$ lower semicontinuity of norms $\overline u$ satisfies
\begin{align}\label{est:apriorilipschitzu}
\|\nabla \overline u\|_{L^\infty(B_\frac12)}
\lesssim&\biggl(\int_{B_{1}}F(\nabla u)\,dx+\|f\|_{L^n(B_1)}^\frac{p}{ p-1}+T^p+\mu^p\biggr)^\frac1p\notag\\
&+\biggl(\int_{B_{1}}F(\nabla u)\,dx+\|f\|_{L^n(B_1)}^\frac{ p}{ p-1}+T^{ p}+\mu^{ p}\biggr)^{\frac1{2p}\frac1{1-\gamma}}\notag\\
&+((T^2+\mu^2)^\frac1{2p}+(T^2+\mu^2)^{\frac12\tilde \gamma })\|f\|_{L^{n,1}(B_1)}^\frac1p
+\|f\|_{L^{n,1}(B_1)}^{\frac1p\frac1{1-\tilde \gamma}}\notag\\
\lesssim&\biggl(\int_{B_{1}}F(\nabla u)\,dx+\|f\|_{L^{n,1}(B_1)}^\frac{p}{ p-1}+T^p+\mu^p\biggr)^\frac1p\notag\\
&+\biggl(\int_{B_{1}}F(\nabla u)\,dx+\|f\|_{L^{n}(B_1)}^\frac{ p}{ p-1}+T^{ p}+\mu^{ p}\biggr)^{\frac1{2p}\frac1{1-\gamma}}
+\|f\|_{L^{n,1}(B_1)}^{\frac1p\frac1{1-\tilde \gamma}},
\end{align}
where we use in the last estimate $\|f\|_{L^n(B_1)}\lesssim\|f\|_{L^{n,1}(B_1)}$ and Youngs inequality in the form 
$$
((T^2+\mu^2)^\frac1{2p}+(T^2+\mu^2)^{\frac12\tilde \gamma })\|f\|_{L^{n,1}(B_1)}^\frac1p\lesssim ((T^2+\mu^2)^\frac1{2}+\|f\|_{L^{n,1}(B_1)}^\frac1{p-1}+\|f\|_{L^{n,1}(B_1)}^{\frac1p\frac1{1-\tilde \gamma}}.
$$
By the definition of $F_\e$ we have
\begin{align*}
\int_{B_1}F_{\e_m}(\nabla u_m)\,dx\geq& \int_{B_1}\widetilde F_{\e_m}(\nabla u_m)\,dx\\
=&\int_{B_{1}}F(\nabla u_m)\,dx+\int_{B_{1}\cap \{|\nabla u|\leq T\}}\widetilde F_{\e_m}(\nabla u_m)-F(\nabla u_m)\,dx\\
\geq&\int_{B_{1}}F(\nabla u_m)\,dx-|B_1|\sup_{z\leq T}|\widetilde F_\e(z)-F(z)|.
\end{align*}
Hence, using convexity of $F$ and the uniform convergence \eqref{ass:tildeFeps1}, we can pass to the limit (along the above chosen subsequence) in \eqref{est:eneum2} and obtain with help of \eqref{claim:umem} and \eqref{claim:umem0}
\begin{align*}
\int_{B_1}F(\nabla \overline u)\,dx\leq \int_{B_{1}}F(\nabla u)-f(u-\overline u)\,dx
\end{align*}
and thus
$$
\mathcal F(\overline u,B_1)\leq \mathcal F(u,B_1).
$$
The above inequality combined with $\overline u\in u+W_0^{1,p}(B_1)$ and the strict convexity of $\mathcal F(\cdot,B_1)$ implies $\overline u=u$. The claimed estimate \eqref{est:apriorilipschitzu0} follows by sending $T$ to $0$ in \eqref{est:apriorilipschitzu} combined with \eqref{ass:Fpq} in the form $\nu\mu^p\leq \fint_{B}F(\nabla u)\,dx$.

\step 3 Construction of $\widetilde F_\e$.

Let $\rho\in C^\infty(\R,[0,1])$ be such that $\rho\equiv1$ on $(-\infty,\frac{T}4)$ and $\rho\equiv0$ on $(\frac{T}3,\infty)$.
We set $\hat F_\e:=F\ast \varphi_{\e}$ where $\varphi_\e$ is as in Step~1 and $\widetilde F_\e=\rho \hat F_\e+(1-\rho)F$. By general properties of the mollification, we have that $\hat F_\e$ is smooth, non-negative and convex, and thus $\widetilde F_\e$ is non-negative, locally $C^2$ and it holds $\widetilde F_\e\equiv F$ on $\R^n\setminus B_{T/2}$.  Since $F$ is convex and locally bounded, we also have \eqref{ass:tildeFeps1}. It remains to show that $\widetilde F_\e$ is convex for $\e>0$ sufficiently small. For this we observe that
$$
\nabla^2 \widetilde F=\rho \nabla^2 \hat F_\e+(1-\rho)\nabla^2 F+(\hat F_\e-F_\e)\nabla^2\rho +\nabla(\hat F_\e-F_\e)\otimes\nabla \rho+\nabla \rho\otimes \nabla(\hat F_\e-F_\e)
$$
is strictly positive definite since $\rho \nabla^2 \hat F_\e+(1-\rho)\nabla^2 F$ is strictly positive definite and the remainder tends to zero as $\e\to0$. 

\step 4 The case $f\equiv 0$. It is straightforward to check that the restriction $\tilde\gamma\in(0,1)$ is not needed in Lemma~\ref{L:apriorilipschitz} if $f\equiv0$ and since we checked in Remark~\ref{rem:gamma} that \eqref{eq:pq} suffices to ensure $\gamma\in(0,1)$ the claim follows.
\end{proof}

\section*{Acknowledgments}

PB was partially supported by the German Science Foundation DFG in context of the Emmy Noether Junior Research Group BE 5922/1-1.

\end{document}